\documentclass[11pt]{article}
\usepackage[utf8]{inputenc}
\usepackage[utf8]{inputenc}
\usepackage[]{datetime}
\usepackage[noadjust]{cite}
\usepackage{amssymb}
\usepackage{graphicx}
\usepackage{enumerate}
\usepackage{amsfonts}
\usepackage{amsmath}
\usepackage{stmaryrd}
\usepackage{xspace}
\usepackage{amsthm}
\usepackage{xcolor}
\usepackage{tikz}
\usepackage{todonotes}
\usepackage{wasysym}

\newcommand{\polylog}{\text{\rm poly\!}\log}

\tikzset{black node/.style={draw, circle, fill = black, minimum size = 5pt, inner sep = 0pt}}
\tikzset{white node/.style={draw, circle, fill = white, minimum size = 5pt, inner sep = 0pt}}
\tikzset{normal/.style = {draw=none, fill = none}}

\newtheorem{theorem}{Theorem}
\newtheorem{lemma}{Lemma}

\newtheorem{proposition}{Proposition}

\newtheorem{observation}{Observation}

\newtheoremstyle{case}
  {3pt}
  {3pt}
  {\normalfont}
  {0pt}
  {\itshape}
  {:}
  { }
  {\thmnote{#3}}
\theoremstyle{case}

\theoremstyle{definition}

\newcommand{\N}{\mathbb{N}}
\newcommand{\intv}[2]{\left \{ #1,\dots, #2 \right \}}

\DeclareMathOperator{\mdelta}{{\mathrm{m}\delta}}
\DeclareMathOperator{\mdeg}{\mathrm{mdeg}}
\DeclareMathOperator{\mult}{\mathrm{\sf mult}}

\DeclareMathOperator{\tw}{\mathbf{tw}} 
\DeclareMathOperator{\tcw}{\mathbf{tcw}} 
\DeclareMathOperator{\tpw}{\mathbf{tpw}} 
\DeclareMathOperator{\cover}{\mathbf{cover}}
\DeclareMathOperator{\pack}{\mathbf{pack}}
\DeclareMathOperator{\desc}{desc}
\DeclareMathOperator{\adh}{adh}

\newcommand{\septhanks}{\ $^,$}

\title{\bf Packing and Covering Immersion Models\\ of Planar subcubic Graphs\thanks{Emails: \scriptsize
{\texttt{archontia.giannopoulou@gmail.com}},\,{\texttt{ojoungkwon@gmail.com}},\,{\texttt{jean-florent.raymond@mimuw.edu.pl}}, {\texttt{sedthilk@thilikos.info}}\,.}}

\date{}

\makeatletter\let\@fnsymbol\@alph\makeatother

\author{Archontia Giannopoulou\thanks{Institute of Informatics, University of Warsaw, Poland. }\septhanks\thanks{Supported by the Warsaw Center of Mathematics and Computer Science.} 
\and 
O-joung Kwon\thanks{Institute for Computer Science and Control, Hungarian Academy of Sciences, Hungary.}\septhanks\thanks{Supported by ERC Starting Grant PARAMTIGHT (No. 280152).}
\and 
Jean-Florent Raymond$^{\text{b},}$\thanks{AlGCo project team, CNRS, LIRMM, Montpellier,
    France.}\septhanks\thanks{Supported by the (Polish)
    National Science Centre grant PRELUDIUM 2013/11/N/ST6/02706.}
\and \\ Dimitrios M. Thilikos$^{\text{f,}}$\thanks{Department of Mathematics, National and Kapodistrian University of Athens, Greece.}
}

\begin{document}

\maketitle

\begin{abstract}
\noindent 
A graph $H$ is an immersion of a graph $G$ if $H$ can be obtained by some subgraph 
$G$ after lifting incident  edges.
We prove that there is a polynomial function $f:\N\times\N\rightarrow\N$, such that 
if $H$ is a connected planar subcubic graph on $h>0$ edges,
$G$ is a graph, and $k$ is a non-negative integer, then  either $G$ contains $k$ vertex/edge-disjoint subgraphs, 
each containing $H$ as an immersion, or $G$ contains a set $F$ of $f(k,h)$ vertices/edges 
such that $G\setminus F$  does  not contain $H$ as an immersion.
\end{abstract}

\noindent {\bf keywords:} Erdős–Pósa properties, Graph immersions, Packings and coverings in graphs

\section{Introduction}


All graphs is this paper are finite, undirected, loopless, and may have multiedges. 
Let ${\cal C}$ be a class of graphs. 
An {\em ${\cal C}$-vertex/edge cover} of $G$ is a set $S$ of vertices/edges such that 
each subgraph of $G$ that is isomorphic to a graph in ${\cal C}$ 
contains some element of $S$.
A {\em ${\cal C}$-vertex/edge packing} of $G$ is a collection of vertex/edge-disjoint subgraphs of $G$, each isomorphic to some graph in ${\cal C}$.

We say that a graph class ${\cal C}$ has the {\em vertex/edge Erdős–Pósa property}  
 (shortly {\em {\sf v/e}-E{\sf \footnotesize  \&}P property}) for some graph class ${\cal G}$ if
there is a function $f:\N \rightarrow \N$, called a {\em gap function},
such that, for every graph $G$ in ${\cal G}$ and every 
non-negative integer $k$,  either $G$ has a vertex/edge ${\cal C}$-packing of size $k$
or $G$ has a  vertex/edge ${\cal C}$-cover of size $f(k)$. 
In the case where ${\cal G}$ is the class of all graphs
we simply say that ${\cal C}$ has the {\sf v/e}-E{\sf \footnotesize  \&}P property.
An interesting topic in  Graph Theory, related to the notion of duality between graph parameters, 
is to detect  instantiations 
of ${\cal C}$ and ${\cal G}$ such that ${\cal C}$ has the  {\sf v/e}-E{\sf \footnotesize  \&}P property  for ${\cal G}$
and, if yes, optimize the corresponding gap. Certainly,  the first result of this type was 
the celebrated result of Erdős and Pósa in~\cite{ErdosP65} who proved that the class of all cycles has the 
{\sf v}-E{\sf \footnotesize  \&}P property with  gap function  $O(k\cdot \log k)$. This result have triggered 
a lot of research on its possible extensions. One of the most general ones  was given 
in~\cite{RobertsonS86GMV} where its was proven 
that the class of graphs that are contractible to some graph $H$
have the {\sf v}-E{\sf \footnotesize  \&}P property iff $H$ is planar  (see also~\cite{Diestel05grap,ChekuriC13larg,ChekuriC13poly} for 
improvements on the gap function).

Other  instantiations of ${\cal C}$  for which  the {\sf v}-E{\sf \footnotesize  \&}P property
has been proved concern 
 odd cycles~\cite{KawarabayashiN2007erdo,RautenbachR01thee}, long cycles~\cite{BirmeleBR07erdo}, and graphs containing cliques as minors~\cite{DiestelKW12thee} (see also~\cite{ReedRST96pack,KakimuraKK12erdo,GeelenK09thee} for results on more general combinatorial structures).  
 
As noticed in~\cite{Diestel05grap}, cycles 
  have  the {\sf e}-E{\sf \footnotesize  \&}P property as well. Interestingly,  
 only few more results exist for the  cases where the  {\sf e}-E{\sf \footnotesize  \&}P property
 is satisfied. It is known for instance that graphs contractible
 to $\theta_{r}$ (i.e. the graph consisting of two vertices and an edge of multiplicity $r$ between them) have the {\sf e}-E{\sf \footnotesize  \&}P property~\cite{ChatzidimitriouRST15mino}. Moreover it was proven that odd cycles 
have the  {\sf e}-E{\sf \footnotesize  \&}P property for planar graphs~\cite{KralV04edge}
and for   4-edge-connected
graphs \cite{KawarabayashiN2007erdo}.

Given two graphs $G$ and $H$, we say that $H$ is an {\em immersion}
of $G$ if $H$ can be obtained from some subgraph of $G$ by lifting 
 incident edges (see Section~\ref{d0our} for the definition of the lift operation). Given a 
graph $H$, we denote by ${\cal I}(H)$
the set of all graphs that contain $H$ as an immersion. 
Using this terminology, the edge variant of the original
 result of  Erdős and Pósa in~\cite{ErdosP65} implies that the class ${\cal I}(\theta_{2})$ has the  {\sf v}-E{\sf \footnotesize  \&}P  property (and, according to~\cite{Diestel05grap}, the 
 {\sf e}-E{\sf \footnotesize  \&}P  property as well).
A natural question is whether this can be extended 
for  ${\cal I}(H)$, for other $H$'s, different than
${\theta_{2}}$. This is the question that we consider in this paper.
A distinct line of research is to identify the graph classes
$\mathcal{G}$ such that for every graph $H$, ${\cal I}(H)$
has the {\sf e}-E{\sf \footnotesize  \&}P property for~$\mathcal{G}$.
In this direction, it was recently proved in~\cite{Liu15pack} that for every graph $H$, ${\cal I}(H)$
has the {\sf e}-E{\sf \footnotesize  \&}P property for 4-edge-connected graphs.
%
%
%
%

In this paper we show that if $H$ is non-trivial (i.e., has at least one edge),
connected, planar,  
and subcubic, i.e., each vertex is incident to at most 3 edges, 
then ${\cal I}(H)$ has the {\sf v/e}-E{\sf \footnotesize  \&}P property (with polynomial gap
in both cases).
More concretely, our main result is the following.

\begin{theorem}
\label{main4u}
Let $H$ be a connected planar subcubic graph of   $h>0$ edges, let $k\in\N$, and let $G$ be a graph without any ${\cal I}(H)$-vertex/edge packing of size greater than $k$. Then $G$ has a  ${\cal I}(H)$-vertex/edge cover of size bounded by a polynomial function of $h$ and $k$.
\end{theorem}

The main tools of our proof are the graph invariants of tree-cut width and 
tree-partititon width, defined in~\cite{Wollan15} and~\cite{DingO96ontr} 
respectively (see Section~\ref{d0our}
for the formal definitions). Our proof uses the fact that every graph
of polynomially (on $k$) big tree-cut width contains a
wall of height $k$ as an immersion (as proved in~\cite{Wollan15}).
This permits us to consider only graphs of bounded  tree-cut width
and, by applying suitable reductions, we finally  reduce the problem to graphs of bounded tree partition
width (Theorem~\ref{gk0opo}). The result follows as we next prove that for every $H$, the class 
 ${\cal I}(H)$ has  the {\sf e}-E{\sf \footnotesize  \&}P property for graphs of bounded tree-partition width (Theorem~\ref{o0p2kz9}).

One might conjecture that the result in~Theorem~\ref{main4u} is tight in the sense 
that both being planar and subcubic are necessary for $H$ in order  ${\cal I}(H)$ to have
the {\sf e}-E{\sf \footnotesize  \&}P property. In this direction, in~Section~\ref{d0opo45iuyt},
we give counterexamples for the cases where $H$ is planar 
but not subcubic and is subcubic but not planar.

\section{Definitions and preliminary results}
\label{d0our}

We use $\N^{+}$ for the set of all positive integers and we set $\N=\N^{+}\cup\{0\}$. 
Given a function $f:A\rightarrow B$ and a set $C\subseteq A$, we denote by $f|_{C}=\{(x,f(x)) \mid x\in C\}$.

\paragraph{\bf Graphs.}

As already mentioned, we deal with loopless graphs where multiedges are allowed. Given a graph $G$, we denote by
$V(G)$ its set of vertices and by $E(G)$ its multiset of edges. The
notation $|E(G)|$ stands for the total number of edges, that is,
counting multiplicities. 
We use the term \emph{multiedge} to refer to a
2-element set of adjacent vertices and the term \emph{edge} to deal
with one particular instanciation of the multiedge
connecting two vertices. The function $\mult_G$ maps a set of two
vertices of $G$ to the multiplicity of the edge connecting them,
or zero if they are not adjacent.
If $\mult_G(\{u,v\}) = k$ for some  $k\in \N^{+}$, we denote by $\{u,v\}_1,
\dots, \{u,v\}_{k}$ the distinct edges connecting
$u$ and~$v$. For the sake of clarity, we identify a multiedge of
multiplicity one and its edge and write $\{u,v\}$ instead of
$\{u,v\}_1$ when~$\mult_G(\{u,v\}) = 1$.

We denote by $\deg_G(v)$ the degree of a
vertex $v$ in a graph $G$, that is, the number of vertices that are
adjacent to~$v$. The multidegree of $v$, that we write $\mdeg_G(v)$,
is the number of edges (i.e.\ counting multiplicities) incident with~$v$. We drop the subscript when it is clear from the
context.

%
%

\paragraph{\bf Immersions.}
Let $H$ and $G$ be graphs. We say that $G$ contains $H$ as an
\emph{immersion} if there is a pair of functions $(\phi, \psi)$, called
an $H$-\emph{immersion model}, such that $\phi$ is an injection of $V(H) \to V(G)$ and
$\psi$ sends $\{u,v\}_i$ to a path of $G$ between
$\phi(u)$ and $\phi(v)$, for every $\{u,v\}\in E(H)$ and every $i \in
\intv{1}{\mult_H(\{u,v\})}$, in a way such that
distinct edges are sent to edge-disjoint paths. Every vertex in the image of $\phi$
is called a \emph{branch vertex}. We will make use of the following 
easy observation.

\begin{observation}\label{branches}
  Let $H$ and $G$ be two graphs, and let $(\phi, \psi)$ be an
  $H$-immersion model in $G$. Then for every vertex $x$ of $G$, we have
  $\mdeg_H(x) \leq \mdeg_G(\phi(x))$.
\end{observation}

An \emph{$H$-immersion expansion} $M$ in a graph $G$ is a subgraph of $G$ defined as
follows: $V(M) = \phi(V(H)) \cup \bigcup_{e \in H} V(\psi(e))$ and $E(M) = \bigcup_{e \in H}E(\psi(e))$
for some $H$-immersion model $(\phi,\psi)$ of~$G$.
We call the paths in $\psi(E(H))$ {\em certifying paths} of the $H$-immersion expansion $M$.

We say that two edges are {\em incident} if they share some
endpoint. A \emph{lift} of two incident edges
$e_1=\{x,y\}$ and $e_2=\{y,z\}$ of $G$ is the operation that removes the
edges $e_{1}$ and $e_{2}$ from the  
graph and then,  if $x\neq z$,  adds the edge $\{y,z\}$ (or increases the multiplicity of $\{y,z\}$ 
if this edge already exists). Notice that $H$ is an immersion of $G$ if and only if a graph isomorphic to $H$ 
can be obtained from some subgraph of $G$ after applying lifts of incident edges%
\footnote{While we mentioned 
this definition in the introduction, we now adopt the more technical definition
of immersion in terms of immersion models as this will facilitate the presentation of the proofs.}.



The \emph{dissolution} of a vertex of degree two of a graph
is the operation of adding an edge joining
its endpoints and then deleting this vertex.

\paragraph{\bf Packings and coverings.}
An $H$-cover of $G$ is a set $C \subseteq E(G)$ such that
$G \setminus C$ does not contain $H$ as an immersion.
An $H$-packing in $G$ is a collection of edge-disjoint $H$-immersion expansions
in $G$.
We denote by $\pack_H(G)$ the maximum size of an $H$-packing
and by $\cover_H(G)$ the minimum size of an $H$-covering in~$G$.

\paragraph{\bf Rooted trees.}
A \emph{rooted tree} is a pair $(T,s)$ where $T$ is a tree and $s \in
V(T)$ is a vertex referred to as the \emph{root}. Given a vertex $x\in V(T)$, the
\emph{descendants} of $x$ in $(T,s)$,
denoted by $\desc_{(T,s)}(x)$, is the set containing each vertex
$w$ such that the unique path from $w$ to $s$ in $T$ contains $x$.
If $y$ is a descendant of $x$ and is adjacent to $x$, then it is a
\emph{child} of~$x$.
Two vertices of $T$ are \emph{siblings} if they are children of the
same vertex.
 Given a
rooted tree $(T,s)$ and a vertex $x\in V(G)$, the \emph{height} of $x$
in $(T,s)$ is the maximum distance between $x$ and a vertex in $\desc_{(T,s)}(x)$.

We now define two types of decompositions of graphs: tree-partitions
(cf. \cite{Seese, Halin1991203})
and tree-cut decompositions (cf.\ \cite{Wollan15}).

\paragraph{\bf Tree-partitions.}
We introduce, especially for the needs of our proof,
a multigraph extension of the parameter of tree-partition width defined in~\cite{Seese, Halin1991203}
where we could consider the number of edges between the bags and the number of vertices in the bags.
A \emph{tree-partition} of a graph $G$ is a pair ${\cal D}=(T, {\cal X})$
where $T$ is a tree and  ${\cal X}=\{X_t\}_{t\in V(T)}$ is a
partition of $V(G)$ such that either $|V(T)|=1$ or for every $\{x,y\}\in
E(G)$, there exists an edge $\{t,t'\}\in E(T)$ where
$\{x,y\}\subseteq X_{t}\cup X_{t'}$. We call the elements of $\mathcal{X}$ {\em bags} of ${\cal D}$. Given an edge $f=\{t,t'\}\in
E(T)$, we define $E_{f}$ as the set of edges with one endpoint in
$X_{t}$ and the other in $X_{t'}$. 
The {\em width} of ${\cal D}$ is
defined as $\max\{|X_{t}|\}_{t\in V(T)}\cup\{|E_{f}|\}_{f\in E(T)}.$
The \emph{tree-partition width} of $G$ is 
the minimum width over all  tree-partitions of $G$ and will be denoted by $\tpw(G)$. 
A \emph{rooted tree-partition} of a graph $G$ is a 
triple ${\cal D}=((T,s), {\cal X})$ where $(T,s)$ is a rooted tree and 
$({\cal X},T)$ is a tree-partition of $G$.
\paragraph{\bf Tree-cut decompositions.}

A \emph{near-partition} of a set $S$ is a collection of pairwise
disjoint subsets $S_1,\dots, S_k \subseteq S$ (for some $k \in \N$)
such that $\bigcup_{i=1}^k S_i = S$. Observe that this definition allows
a set of the familly to be empty.
A \emph{tree-cut decomposition} of a graph $G$ is a pair ${\cal D}=(T, {\cal X})$
where $T$ is a  tree and ${\cal X}=\{X_t\}_{t\in V(T)}$ is a
near-partition of $V(G)$. 
As in the case of
tree-partitions, we call the elements of $\mathcal{X}$ {\em bags} of ${\cal D}$.
A \emph{rooted tree-cut decomposition} of a graph $G$ is a 
triple ${\cal D}=((T,s), {\cal X})$ where $(T,s)$ is a rooted tree and 
$({\cal X},T)$ is a tree-cut decomposition of $G$.
Given that ${\cal D}=((T,s), {\cal X})$ is a rooted tree partition or a rooted tree-cut decomposition of $G$ 
and given  $t \in
V(T)$, 
we set $G_t = G\left
    [ \bigcup_{t \in \desc_{(T,s)}(t)}X_t\right ].$ 

The \emph{torso} of a tree-cut decomposition $(T, \mathcal{X})$ at a node $t$ is the graph obtained from $G$ as follows. 
If $V(T)=\{t\}$, then the torso at $t$ is $G$. Otherwise let $T_1, \ldots, T_{\ell}$ be the connected components of $T\setminus t$. 
The torso $H_t$ at $t$ is obtained from $G$ by \emph{consolidating} each vertex set $\bigcup_{b\in V(T_i)} X_b$ into a single vertex $z_i$. The operation of consolidating a vertex set $Z$ into $z$ is to replace $Z$ with $z$ in $G$, and for each edge $e$ between $Z$ and $v\in V (G)\setminus Z$, adding an edge $zv$ in the new graph. 
Given a graph $G$ and $X\subseteq  V(G)$, let the \emph{3-center} of $(G, X)$ be the unique graph obtained from $G$ by suppressing vertices in $V (G) \setminus  X$ of degree two and deleting vertices of degree $1$. For each node $t$ of $T$, we denote by $\widetilde{H_t}$ the \emph{3-center} of $(H_t,X_t)$, where $H_t$ is the torso of $(T, \mathcal{X})$ at $t$.

Let ${\cal D}=((T,s), {\cal X})$ be a rooted tree-cut decomposition of~$G$.
The adhesion of a vertex $t$ of $T$, that we write
$\adh_\mathcal{D}(t)$, is the number of edges with exactly one
endpoint in~$G_t$. 
 The \emph{width} of a tree-cut decomposition $(\mathcal{X}, T)$ of $G$ is $\max_{t\in V (T)} \{\adh_\mathcal{D}(t),|\widetilde{H_t}|\}$. 
 The \emph{tree-cut width} of $G$, denoted by $\tcw(G)$, is the minimum width over all tree-cut decompositions of $G$.

A vertex $t\in V(T)$ is {\em thin} if $\adh_\mathcal{D}(t)\leq 2$, and {\em bold} otherwise.
We also say that ${\cal D}$
is \emph{nice} if for every thin vertex $t \in V(T)$
we have $N(V(G_t)) \cap \bigcup_{b\ \text{\rm is a sibling of}\ t}V(G_b)
  = \emptyset.$
In other words, there is no edge from a vertex of $G_t$ to a vertex of
$G_b$, for any sibling $b$ of~$t$, whenever $t$ is thin.
The notion of nice tree-cut decompositions has been introdued by
Ganian et al.\ in~\cite{ganian2015}. Furthermore, they proved the
following result.
\begin{proposition}[\cite{ganian2015}]\label{ganian}
  Every rooted tree-cut decomposition can be transformed into a nice one
  without increasing the width.
\end{proposition}

We say than an edge of $G$ \emph{crosses} the bag $X_t$, for some $t
\in V(T)$ if its endpoints belongs to bags $X_{t_1}$ and $X_{t_2}$,
for some $t_1, t_2\in V(T)$ such that $t$ belongs to the interior of
the (unique) path of $T$ connecting $t_1$ to~$t_2$.

\section{From tree-cut decompositions to tree-partitions}
\label{sec:tc2tp}

The purpose of this section is to prove the following theorem.
\begin{theorem}
\label{gk0opo}
  For every connected graph
   $G$, and every connected graph $H$ with  at least one edge,
  there is a graph $G'$ such that
  \begin{itemize}
  \item $\tpw(G') \leq (\tcw(G)+1)^{2}/2$, 
  \item $\pack_{H^+}(G') \leq \pack_H(G)$, and 
  \item $\cover_{H}(G) \leq \cover_{H^+}(G')$.
  \end{itemize}
\end{theorem}

Theorem~\ref{gk0opo} will allow us in Section~\ref{sec:linep} to consider
graphs of bounded tree-partition width instead of graphs of bounded
tree-cut width. Before we proceed with the proof of Theorem~\ref{gk0opo}, we need some 
definitions and a series of auxiliary results.
\medskip

For every graph $G$, we define $G^+$ as the graph obtained if, for
every vertex $v$, we add two new vertices $v'$ and $v''$ and the edges
$\{v', v''\}$ (of multiplicity 2), $\{v,v'\}$ and $\{v,v''\}$ (both of
multiplicity 1). Observe that for
every $G$, we have~$\mdelta(G^+) \geq 3$.
We also define $G^*$ as the graph obtained by adding, for every vertex
$v$, the new vertices $v_1', \dots, v'_{\mdeg(v)}$ and $v_1'', \dots,
v''_{\mdeg(v)}$ and the edges $\{v_i', v_i''\}$ (of multiplicity 2),
$\{v, v_i'\}$, and $\{v, v_i''\}$ (both of multiplicity 1), for
every~$i \in \intv{1}{\deg(v)}$. If $v$ is a vertex of $G$, then we
denote by $Z_{v,i}$ the subgraph $G^*[\{v,v_i',v_i''\}]$ for every $i
\in \intv{1}{\mdeg_G(v)}$.

 Our first aim is to prove the following three lemmata.

\begin{lemma}\label{subdequal}
  Let $G$ be a graph, let $H$ be a connected graph with at least one edge and
  let $G'$ be a subdivision of~$G^*$. Then we have
  \begin{itemize}
  \item $\pack_{H^+}(G^*) = \pack_{H^+}(G')$ and
  \item $\cover_{H^+}(G^*) = \cover_{H^+}(G')$.
  \end{itemize}
\end{lemma}

\begin{proof}
We denote by $S$ the set of subdivision vertices added during the
construction of $G'$ from~$G^+$.
As $G'$ is a subdivision of $G^*$, we have $\pack_{H^+}(G') \geq
\pack_{H^+}(G^*)$ and $\cover_{H^+}(G') \geq \cover_{H^+}(G^*)$.

As a consequence of Observation~\ref{branches} and the fact that $\mdelta(H^+)
\geq 3$, if $M$ is an $H^+$-immersion expansion in $G'$ then no branch vertex of
$M$ belongs to~$S$. Indeed, every vertex of $S$ has multidegree 2
in~$G'$. Therefore, by dissolving in $M$ the vertices of $S$ that belong to
$V(M)$, we obtain an $H^+$-immersion expansion in~$G^*$. It follows that
$\pack_{H^+}(G^*) \geq \pack_{H^+}(G')$, hence $\pack_{H^+}(G^*) =
\pack_{H^+}(G')$.

On the other hand, let $X$ be an $H^+$-cover of $G^*$ and let $X'$ be
a set of edges constructed by taking, for every $e \in X$, an edge of
the path of $G'$ connecting the endpoints of $e$ that has been created by
subdividing~$e$. Assume that $X'$ is not an $H^+$-cover of
$G'$. According to the remark above, this implies that $X$ is not an
$H^+$-cover of $G^*$, a contradiction. Hence $X'$ is an $H^+$-cover of
$G'$ and thus $\cover_{H^+}(G^*) = \cover_{H^+}(G')$.
\end{proof}

\begin{lemma}\label{pabound}
  For every two graphs $H$ and $G$ such that $H$ is connected and has
  at least one edge, we have
  $\pack_{H^+}(G^*) \leq \pack_{H}(G)$.
\end{lemma}

\begin{proof}
In $G^*$ (respectively $H^+$), we say that a vertex is \emph{original}
if it belongs to $V(G)$ (respectively $V(H)$).
Let $(\phi, \psi)$ be an $H^+$-immersion model in $G^*$.

We first show that if $u$ is an original vertex of $H^+$, then $\phi(u)$
is an original vertex of~$G^*$.
By contradiction, let us assume that $\phi(u)$ is not original, for
some original vertex $u$ of $H^+$. Then $\phi(u) = v'_i$ or $\phi(u) = v''_i$,
for some $v\in V(G)$ and $i \in \intv{1}{\mdeg_G(v)}$.

Observe that since $H$ is connected and has at least one edge, every
vertex of $H^+$ has degree at least three: let $x$, $y$, and $z$ be the endpoints
of three multiedges incident with~$u$. Then
$\psi(\{u,x\})$, $\psi(\{u,x\})$, and $\psi(\{u,x\})$ 
are edge-disjoint paths connecting $\phi(u)$ to three distinct vertices. This is
not possible because there is an edge cut of size two,
$\{\{v,v'_i\},\{v,v''_i\}\}$, separating the two vertices $v'_i$ and $v''_i$
(among which is $\phi(u)$) from the rest of the graph.
Consequently, if $u \in V(H^+)$ is original, then $\phi(u)$ is original.

Let us now consider an edge $\{u,v\} \in E(H)$. By the above remark,
$\phi(u)$ and $\phi(v)$ are original vertices of~$G^*$. It is easy to
see that $\psi(\{u,v\})$ contains only original vertices
of~$G^*$. Indeed, if this path contained a non-original vertex $w'$ or
$w''$ for some original vertex $w$ of $V(G^*)$, it would use $w$ twice in order to
reach $u$ and $v$, what is not allowed.
Therefore,  from the definition of $H^+$, the pair  $(\phi|_{V(H)}, \psi|_{E(H)})$  
is an $H$-immersion model of 
$G$. 

We proved that every $H^+$-immersion-expansion of $G^*$ contains an $H$-immersion-expansion that
belongs to the subgraph $G$ of $G^*$. Consequently every $H^+$-packing
of $G^*$ contains an $H$-packing of the same size that belongs to $G$,
and the desired inequality follows.
\end{proof}

\begin{lemma}\label{cobound}
  For every two graphs $H$ and $G$ such that $H$ is connected and has
  at least one edge, we have $\cover_H(G) \leq \cover_{H^+}(G^*)$.
\end{lemma}

\begin{proof}
Similarly to the proof of Lemma~\ref{pabound}, we say that an edge of
$G^*$ is \emph{original} if it belongs to~$E(G)$.
Let $X\subseteq E(G^*)$ be a minimum cover of $H^+$-immersion expansions in~$G^*$.
\smallskip

\noindent \textit{First case:} all the edges in $X$ are original.
In this case, $X$ is an $H$-cover of $G$ as well. Indeed, if
$G\setminus X$ contains an $H$-immersion expansion $M$, then $G^*\setminus X$ contains
$M^*$ that, in turn, contains~$H^+$. Hence in this case, $\cover_H(G) \leq \cover_{H^+}(G^*)$.
\smallskip

\noindent \textit{Second case:} there is an edge $e \in X$ that is not
original.
Let $v$ be the original vertex of $G^*$ such that either $e \in
Z_{v,l}$ for some~$l \in \intv{1}{\mdeg_G(v)}$.
Let us first show the following claim.

\noindent{\em Claim:}
  For every $i \in \intv{1}{\mdeg_G(v)}$, there is an edge of
  $Z_{v,i}$ that belongs to $X$.
\smallskip

\noindent{\em Proof of claim:}
  Looking for a contradiction, let us assume that for some $i \in
  \intv{1}{\mdeg_G(v)}$, we have $E(Z_{v,i}) \cap X =
  \emptyset$. Clearly $i \neq l$.
  By minimality of $X$, the graph $G \setminus (X \setminus \{e\})$
  contains an $H^+$-immersion expansion $M$ that uses~$e$. Observe that $M' = M
  \setminus E(Z_{v,l}) \cup E(Z_{v,i})$ contains an $H^+$-immersion expansion
  (since $Z_{v,l}$ and $Z_{v,i}$ are isomorphic). Hence, $M'$ is a subgraph
  of $G \setminus (X \setminus \{e\})$ that contains an
  $H^+$-immersion expansion. This is not possible as $X$ is a cover, so we reach
  the contradiction we were looking for and the claim holds.

We build a set $Y$ as follows. For every edge $f\in X$, if $f$ is
original then we add to~$Y$. Otherwise, if $v_f$ is the (original)
vertex of $G^*$ such that $e \in E(Z_{v_f, i})$ for some $i \in
\intv{1}{\mdeg_G(v_f)}$, then we add to $Y$ all edges that are
incident to $v_f$.

The above claim ensures that when a non-original edge $f$ of $X$ is
encountered, then $X$ contains an edge in each of $Z_{v_f,1}, \dots,
Z_{v_f,\mdeg_G(v_f)}$. Therefore, the same set of edges, of size
$\mdeg_G(v_f)$, will be added to $Y$ when encountering an other edge from $Z_{v_f,1}, \dots,
Z_{v_f,\mdeg_G(v_f)}$. Consequently, $|X| = |Y|$.

Let us not show that $Y$ is an $H^+$-cover of~$G^*$. Suppose that there exists 
an $H^+$-immersion expansion $M$ in~$G^* \setminus Y$. Observe that since $H$ is
connected and has at least one edge, $M$ does not belong to
$\bigcup_{i\in \{1,\ldots,\mdeg_G(u)\}} Z_{u,i}$, for every original vertex $u$
of~$G^*$. Let $$Z = \bigcup_{u \in V(G)} \bigcup_{i\in \{1,\ldots,\mdeg_G(u)\}} E(Z_{u,i})$$
Then $M$ is a subgraph of $G \setminus (Y \cup Z)$. As $X \subseteq Y
\cup Z$, this contradicts the fact that $X$ is a cover. Therefore, $Y$
is an $H^+$-cover. Moreover all the edges in $Y$ are original. As this
situation is treated by the first case above, we are done.
\end{proof}

We are now ready to prove~Theorem~\ref{gk0opo}.

\begin{proof}[Proof of Theorem~\ref{gk0opo}] Let $k=\tcw(G)$.
We examine the nontrivial case where $G$ is not a tree, i.e., $\tcw(G)\geq 2$.
  Let us consider the graph~$G^*$. We claim that $\tcw(G^*) = \tcw(G)$. Indeed, starting from an optimal tree-cut
  decomposition of $G$, we can, for every vertex $v$ of $G$ and for every $i
  \in \intv{1}{\mdeg_G(v)}$, create a bag
  that is a children of the one of $v$ and contains~$\{v_i', v_i''\}$.
  According to the definition of $G^*$, this creates a tree-cut
  decomposition $\mathcal{D}=
((T,s), \{X_t\}_{t \in V(T)})$ of~$G^*$.    Observe that for every vertex $x$ that we introduced to the tree of the decomposition during this
  process, $\adh_{\mathcal{D}}(x) = 2$ and the corresponding bag has
size two. This proves that $\tcw(G^*) \leq
  \max(\tcw(G),2)=\tcw(G)$. As $G$ is a subgraph of $G^{*}$, we obtain $\tcw(G)\leq \tcw(G^*)$ and the proof of the claim is complete.

According to Proposition~\ref{ganian}, we can
assume that $G^*$ has a nice rooted tree-cut decomposition of width $\leq k$. For notational simplicity 
we again denote it  by ${\cal D}=
((T,s), \{X_t\}_{t \in V(T)})$ and,  obviously, we  
can also assume that all leaves of $T$ 
correspond to non-empty bags.

Our next step is to transform the rooted tree-cut decomposition $\mathcal{D} $ into a rooted tree-partition $\mathcal{D'} = ((T,s), \{X'_t\}_{t \in V(T)})$
of a subdivision $G'$ of $G^*$. Notice that the only differences
between two decompositions are that,  in
a tree-cut decomposition, empty bags are allowed as well as edges connecting vertices of bags
corresponding to non-adjacent vertices of $T$.

We  proceed as follows: if
$X$ is a bag crossed by edges, we subdivide every edge crossing $X$
and add the obtained subdivision vertex to~$X$. By repeating this process we
decrease at each step the number of bags crossed by edges, that
eventually reaches zero. Let $G'$ be the obtained graph and observe that $G'$ is 
a subdivision of $G$. As $G$ is connected, the obtained rooted tree-cut decomposition
$\mathcal{D'} = ((T,s), \{X'_t\}_{t \in V(T)})$  is a rooted tree
partition  of $G'$. 

Notice that the adhesion of any bag of $T$ in $\mathcal{D}$ is the same as in $\mathcal{D'}$.
However, the bags of $\mathcal{D}'$ may grow during the construction
of~$G'$. 
Let $t$ be a vertex of $T$ and let $\{t_1, \ldots, t_m\}$ be the set of children of $t$. 
We claim that $|X_t'|\le (k+1)^2/2$.



Let $E_t$ be the set of edges crossing $X_t$ in $G$. Let $H_t$ be the torso of $\mathcal{D}$ at $t$, and let $H_t'=H_t\setminus X_t$. Observe that $|E_t|$ is the same as the number of edges in $H_t'$. Let $z_p$ be the vertex of $H_t'$ corresponding to the parent of $t$, and similarly for each $i\in \{1, \ldots, m\}$ let $z_i$ be the vertex of $H_t'$ corresponding to the child $t_i$ of $t$. Notice that if $t_i$ is a thin child of $t$, then $z_i$ can be adjacent to only $z_p$ as $\mathcal{D}$ is a nice rooted tree-cut decomposition. Thus the sum of the number of incident edges with $z_i$ in $H_t'$ for all thin children $t_i$ of $t$ is at most $\adh_{\cal D}(t) \leq k$. On the other hand, if $t_i$ is a bold child of $t$, then $z_i$ has at least $3$ neighbors in $H_t$, and thus
it is contained in the 3-center of $(H_t, X_t)$. Thus, the number of all bold children of $t$ is bounded by $k-|X_t|$. Since each vertex in $H_t'$ is incident with at most $k$ edges, the total number of edges in $H_t'$ is at most $(k-|X_t|+1)k/2 + k$. As $|E(H_t')|=|E_t|=|X_t'\setminus X_t|$, it implies that $|X_t'|\leq  |X_{t}|+k\cdot (k-|X_{t}|+2)/2 \leq  \max\{2k,k(k+2)/2\} \leq  (k+1)^2/2.$ We conclude that $G'$ has a rooted tree partition  of width at most $(\tcw(G)+1)^{2}/2$.

Recall that $G'$ is a subdivision of  $G^*$. By the virtue of Lemmata~\ref{cobound},
\ref{pabound}, and \ref{subdequal}, we obtain that $\pack_{H^+}(G') \leq \pack_H(G)$ and
   $\cover_{H}(G) \leq \cover_{H^+}(G')$. Hence $G'$ satisfies the desired properties.
\end{proof}

\section{Erdős-Pósa for bounded tree-partition width}
\label{sec:linep}

%

Before we proceed, we require the following lemma and an easy observation.

\begin{lemma}\label{intersect}
 Let $G$ and $H$ be two graphs and let $X \subseteq V(G)$. Let 
  $\mathcal{C}$ be the collection of connected components of $G
  \setminus X$. If $M$ is an $H$-immersion expansion
  of $G$ 
  then 
   $M$ contains vertices
  from at most $(|X|+1)\cdot |E(H)|$ graphs of $\mathcal{C}$.
\end{lemma}

\begin{proof}
Let  $P$ be a certifying  path of $M$ connecting two branch vertices
of $M$.
 Since $P$ is a path, it cannot use twice the same vertex of
 $X$. Besides, as $X$ is a separator, $P$ must go through a vertex of
 $X$ in order to go from one graph of $\mathcal{C}$ to an other
 one. Therefore, $P$ contains vertices from at most $|X|+1$ graphs of~$\mathcal{C}$.
 The desired bound follows as $E(M)$ is partitioned into $|E(H)|$ certifying paths.
\end{proof}

%
%


\begin{observation}
\label{obd4}
Let $G$ and $H$  be graphs and let $F\subseteq E(G)$. Then it holds that $\cover_H(G)\leq \cover_H(G\setminus F)+|F|$.
\end{observation}

\noindent For a graph $H$, we define $\omega_{H}:\N\rightarrow\N$ so that
 $\omega_H(r) = \left \lceil r\cdot
\frac{3r +1}{2} \cdot |E(H)|\right \rceil.$ The next Theorem is an important ingredient of our results.

\begin{theorem}
\label{o0p2kz9}
 Let $H$ be a connected graph with at least one edge. Then for every graph $G$  it holds
that $\cover_H(G) \leq \omega_{H}(\tpw(G))\cdot \pack_H(G)$
\end{theorem}

\begin{proof}
Let us show by induction on $k$ that if $\pack_H(G)\leq k$  and  $\tpw(G) \leq r$
then
$\cover_H(G) \leq \omega_{H}(r) \cdot k$.

The case $k=0$ is trivial. Let us now assume that $k\geq 1$ and that for
every graph $G$ of tree-partition width at most $r$ and such that
$\pack_H(G)=k-1$, we have $\cover_H(G) \leq \omega_{H}(r) (k-1)$.
Let $G$ be a graph such that $\pack_H(G)=k$ and $\tpw(G) \leq r$.
Let also $\mathcal{D} = ((T,s), \{X_t\}_{t \in V(T)})$ be an optimal rooted
tree partition  of~$G$. We say that a vertex $t \in V(T)$ is
\emph{infected} if $G_t$ contains an~$H$-immersion expansion.
Let $t$ be an infected vertex of $T$ of minimum
height. 

\noindent {\em Claim:} If some of the $H$-immersion expansions of $G$ shares an edge  
with  $G_{t'}$ for  some child $t'$ of $t$, then it also shares 
and edge with $E_{\{t,t'\}}$. \smallskip

\noindent {\em Proof of claim:}
Let $M$ be some $H$-immersion expansions of $G$. 
Notice that, by the choice of $t$, $M$
cannot be entirely inside in  $G_{t'}$. This fact, together with the 
connectivity of $M$, implies that $E(M)\cap  E_{\{t,t'\}}\neq\emptyset$.

Suppose that $M$ be an $H$-immersion expansion of $G_t$
and let $U$ be the set of children of $t$ corresponding to bags which share
vertices with $M$. We define the multisets 
$A=E(G[X_t])$, $B=\bigcup_{t'\in U}  E_{\{t,t'\}}$
and $C=\bigcup_{t'\in U}  E(G_t')$. We also set 
$D=A\cup B$.
By the definition of $U$, it follows that 
$E(M)  \subseteq   C\cup D~{\bf (1)}.$

Let us upper-bound the size of $|D|$.
Applying Lemma~\ref{intersect} for $G_{t}$,  $H$, and $X_{t}$,
we have 
$|U| \leq (r+1) \cdot |E(H)|$, hence
$|B| \leq r (r+1) \cdot
|E(H)|$. Besides, every path of $M$ connecting two branch vertices meets
every vertex of $X_t$ at most once (as it is a path), thus $E(M)$ does
not contain an edge of $G[X_t]$ with a multiplicity larger than
$|E(H)|$. It follows that $|A| \leq
\frac{r(r-1)}{2}\cdot |E(H)|$ and finally we obtain that
$|D|=|A|\cup |B| \leq r\cdot \frac{3r +1}{2} \cdot |E(H)|=\omega_H(r).$

Let $G' = G\setminus D$. We now show that $\pack_H(G') \leq  k-1$. Let us consider an
$H$-immersion expansion $M'$ in $G'$. As  $E(M')\subseteq E(G)\setminus D$, if follows that 
$E(M')\cap D  =  \emptyset.~{\bf (2)}$.

Recall that  $B\subseteq D$, which together with~{\bf (2)} implies that $E(M')\cap B=\emptyset$. This fact, 
 combined with the claim above, implies that
$E(M')\cap C = \emptyset.~{\bf (3)}$
From~{\bf (2)} and~{\bf (3)},  we obtain that $E(M')\cap (C\cup D)=\emptyset$, which, combined with~{\bf (1)}, implies that 
$E(M)\cap  E(M)'\neq \emptyset$.
Consequently, every maximum packing of $H$-immersion expansions in $G'$ is
edge-disjoint from $M$. If such a packing had size $\geq k$, it would
form together with $M$ a packing of size $k+1$ in $G$, a
contradiction. Thus $\pack_H(G') \leq  k-1$, as desired.
By the induction hypothesis applied on $G'$, $\cover_{H}(G')\leq \omega_{H}(r) \cdot (k-1)$ edges.
Therefore, from Observation~\ref{obd4}, $\cover_{H}(G)\leq |C|+\cover_{H}(G')\leq |C|+\omega_{H}(r) \cdot (k-1)\leq \omega_{H}(r) \cdot k$ edges as required.
\end{proof}
%

\noindent We  set $\sigma:\N\rightarrow\N$ where $\sigma(r)=\left \lceil \frac{1}{8}(3(r+1)^4+2(r+1)^2)\right \rceil$.

\begin{theorem}
\label{f90op2c}
 Let  $H$ be a connected graph with at least one edge,  $r \in \N$, and $G$ be a graph  where $\tcw(G)\leq r$. 
 Then  $\cover_H(G) \leq \sigma(r)\cdot (4\cdot |V(H)|+|E(H)|)\cdot \pack_H(G)$.
\end{theorem}

\begin{proof} Clearly, we can assume that $G$ is connected, otherwise we work on each of its connected components separately.
By Theorem~\ref{gk0opo}, there is a graph $G'$ where 
 $\tpw(G') \leq (r+1)^{2}/2$, $\pack_{H^+}(G') \leq \pack_H(G)$ and
  $\cover_{H}(G) \leq \cover_{H^+}(G')$. The result follows as, from Theorem~\ref{o0p2kz9},
 $ \cover_{H^+}(G')\leq \omega_{H^+}((r+1)^{2}/2) \cdot \pack_{H^+}(G')$
 and $\omega_{H^+}((r+1)^{2}/2)=
\sigma(r)\cdot |E(H^+)|\leq 
\sigma(r)\cdot (4\cdot |V(H)|+|E(H)|)$. 
\end{proof}

\section{Erdős-Pósa for immersions of subcubic planar graphs}
\label{hj3p0}

\paragraph{\bf \bf Grids and Walls.} Let $k$ and $r$ be positive integers where $k,r\geq 2$. The
\emph{$(k\times r)$-grid} $\Gamma_{k,r}$ is the Cartesian product of two paths of
lengths $k-1$ and $r-1$ respectively. We denote by $\Gamma_{k}$ the $(k\times k)$-grid.
The {\em $k$-wall} $W_{k}$ is the graph obtained from a 
$((k+1)\times (2\cdot k+2))$-grid with vertices $(x,y)$,
$x\in\{1,\dots,k+1\}$, $y\in\{1,\dots,2k+2\}$, after the removal of the
``vertical'' edges $\{(x,y),(x+1,y)\}$ for odd $x+y$, and then the removal of
all vertices of degree 1.

\begin{figure}[h]
\begin{center}
\scalebox{0.6}{\includegraphics{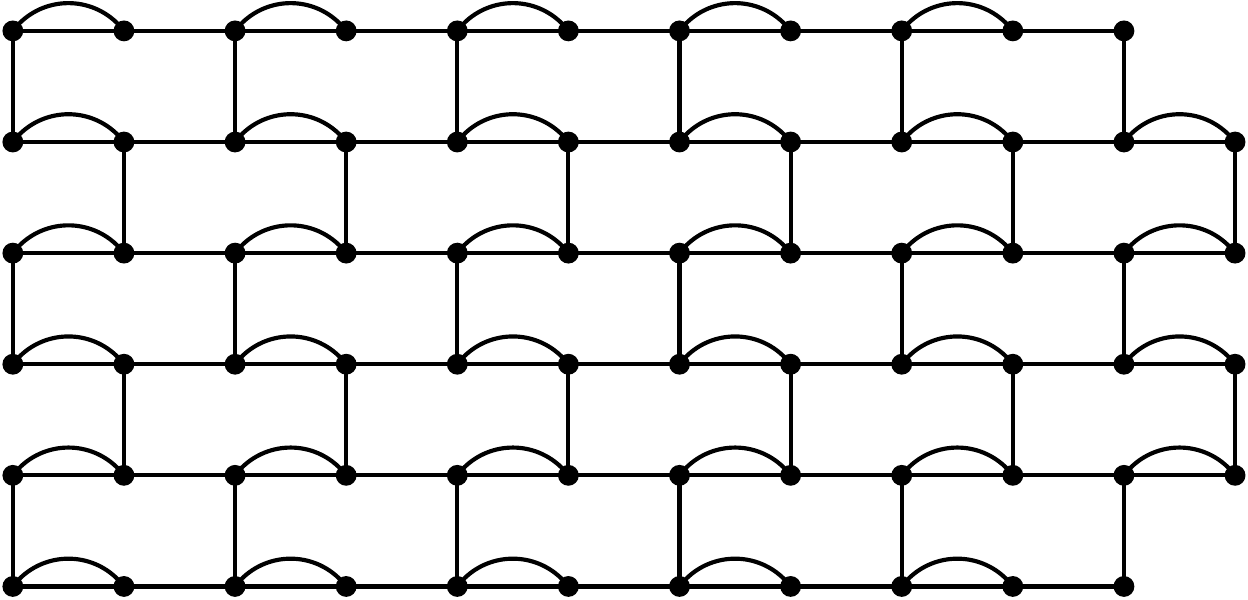}}
\end{center}
\caption{The graph $W^{+}_{5}$.}
\label{fig:enriched-wall}
\end{figure}

Let $W_{k}$ be a wall. We denote by $P^{(v)}_{j}$ the shortest path connecting vertices $(1,2j)$ and $(k+1,2j)$, $j\in [k]$ and call 
these paths the {\em vertical paths of $W_{k}$}, with the assumption that $P^{(v)}_{j}$ contains only vertices $(x,y)$ with $y=2j,2j-1$. 
Note that these paths are vertex-disjoint.
Similarly, for every $i\in [k+1]$ we denote by $P^{(h)}_{i}$ the shortest path connecting vertices $(i,1)$ and $(i,2k+2)$ (or $(i,2k+1)$ if $(i,2k+2)$ has been removed) 
and call these paths the {\em horizontal paths of $W_{k}$}.
Let $E=\{e\mid e\in E(P^{(v)}_{j})\cap E(P^{(h)}_{i}), j\in \{1,2,\dots,k\}, i\in \{1,2,\dots,k+1\}\}$. 
We obtain $W^{+}_{k}$ by $W_{k}$ by adding a second copy of every edge in $E$. (For an example, see Figure~\ref{fig:enriched-wall}.)

\paragraph{\bf Strong immersions.}

If in the definition on when a graph $G$ contains a graph $H$ as an immersion we 
additionally demand  that no branch vertex is an internal vertex of any
certifying path, then the function $(\phi,\psi)$ is an $H$-{\em strong-immersion model} and we say that $G$ contains $H$ as a strong immersion.

\begin{figure}[h]
\begin{center}
\scalebox{0.6}{\includegraphics{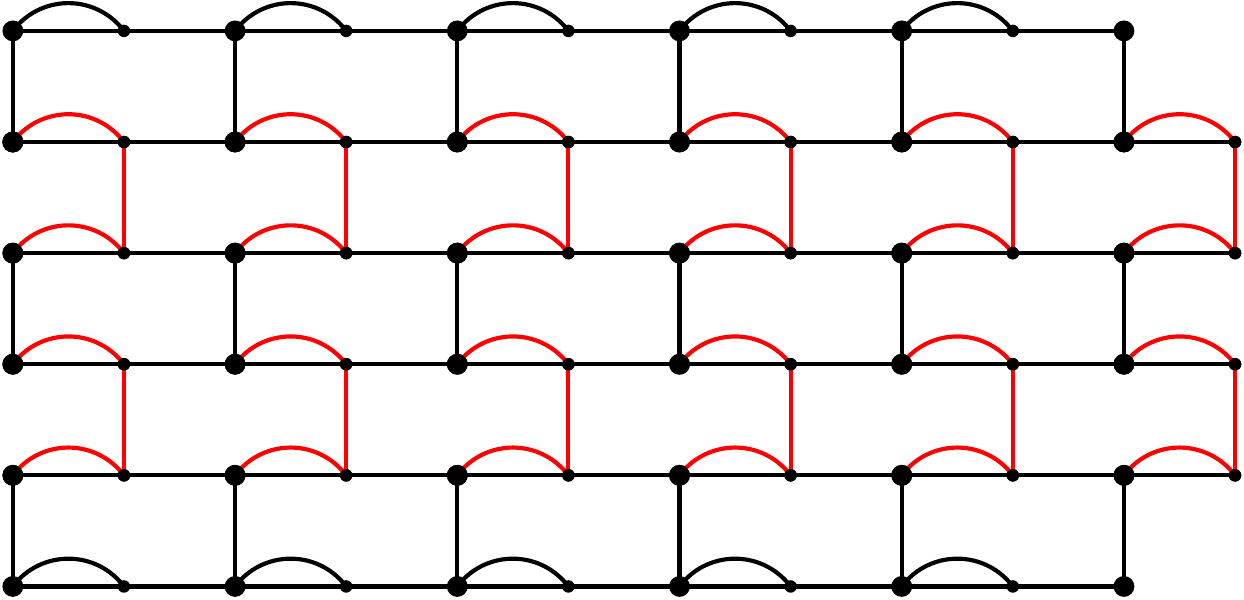}}
\end{center}
\caption{Finding a subdivision of $\Gamma_{k}$ as a strong immersion in $W^{+}_{k}$.}
\label{fig:grid-wall-im}
\end{figure}

\begin{observation}\label{argerglklk}
For every integer $k\geq 2$, there is a $\Gamma_{k}$-strong-immersion model $(\phi,\psi)$ in $W^{+}_{k}$ 
such that $\phi(V(\Gamma_{k}))= \{(i,2j+1)\mid 1\leq i\leq k+1, 0\leq j\leq k\}$ is the set of its branch vertices.
\end{observation}

Figure~\ref{fig:grid-wall-im} illustrates how we may find a subdivision of $\Gamma_{k}$ as a strong immersion in~$W^{+}_{k}$.
We also need the following result.

\begin{lemma}[\cite{Kant96}]\label{aergae}
Every simple planar subcubic graph of $n$ vertices is a topological minor of the $\lfloor \frac{n}{2}\rfloor$-grid.
\end{lemma}

\begin{lemma}\label{lem:subcubicplanar}
Every connected planar subcubic graph $H$  is an immersion of  the wall $W_{|V(H)|}$.
\end{lemma}

\begin{proof}
Let $H$ be a connected planar subcubic graph and let $H'$ be the simple subcubic planar graph obtained from $H$ by subdiving
all but one copies of every multiple edge. Since $H$ is connected we may assume that at least $|V(H)|-1$ edges do not get subdivided.
Noice that $2|E(H)|\leq 3|V(H)|$ and thus $|E(H)|\leq \frac{3}{2}|V(H)|$. Since we add at most $|E(H)|-(V(H)-1)$ vertices to obtain $H'$ from $H$, 
it follows that 
\begin{equation} |V(H')|\leq 2|V(H)|.
\end{equation} 
Let $n=|V(H')|$. By definition, we obtain that $H$ is a topological minor of $H'$ and Lemma~\ref{aergae} yields that $H'$ is a topological minor of
$\Gamma_{\lfloor \frac{n}{2} \rfloor}$. Then $H$ is a topological minor of $\Gamma_{\lfloor \frac{n}{2} \rfloor}$ and let $(\phi,\psi)$
be an $H$-topological-minor model in $\Gamma_{\lfloor \frac{n}{2} \rfloor}$.
Let $H''$ denote the $H$-immersion expansion in $\Gamma_{\lfloor \frac{n}{2} \rfloor}$.
Notice that since $H$ is planar subsubic so is $H''$.

Let $(\phi',\psi')$ be the $\Gamma_{\lfloor \frac{n}{2} \rfloor}$-strong-immersion model in $W^{+}_{\lfloor \frac{n}{2} \rfloor}$
where $$\phi'(V(\Gamma_{\lfloor \frac{n}{2} \rfloor}))\subseteq \{(i,2j+1)\mid 1\leq i\leq \lfloor \frac{n}{2} \rfloor+1, 0\leq j\leq \lfloor \frac{n}{2} \rfloor\}$$ (Observation~\ref{argerglklk}) and notice that its restriction to $H''$ yields an
$H''$-strong-immersion model in $W^{+}_{\lfloor \frac{n}{2} \rfloor}$. 
Let $W'$ be the $H''$-immersion expansion in $W^{+}_{\lfloor \frac{n}{2} \rfloor}$.

To prove our lemma it is enough to show that $W'$ contains a strong immersion model of $H''$ such that its $H''$-immersion expansion $W$ is simple as then 
$W\subseteq W_{\lfloor \frac{n}{2} \rfloor}$.

\begin{figure}[h]
\begin{center}
\scalebox{0.6}{\includegraphics{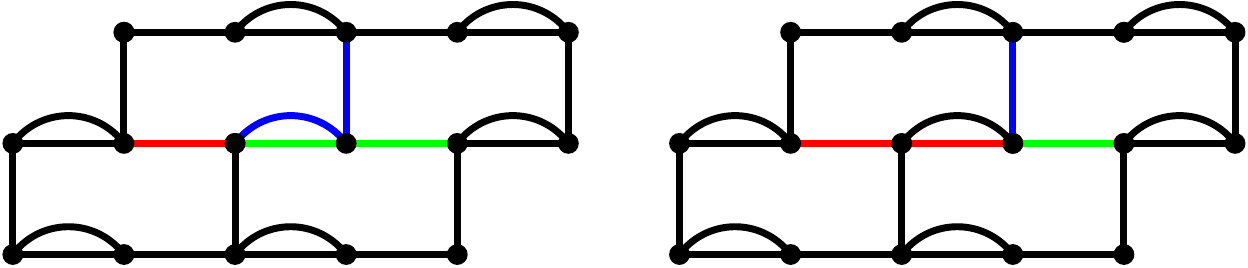}}
\end{center}
\caption{Swapping branch vertices.}
\label{fig:simple-wall-ex}
\end{figure}

We first show that if $W'$ contains a vertex $z$ of $\phi'(V(\Gamma_{\lfloor \frac{n}{2} \rfloor}))$ then $z\in \phi'(V(H''))$.
Let us assume, to the contrary, that $W'$ contains a vertex $z$ such that $z\in \phi'(V(\Gamma_{\lfloor \frac{n}{2} \rfloor}))\setminus \phi'(V(H''))$.
Then $z$ is internal to one of the certifying paths of the $H''$-immersion expansion in $W^{+}_{\lfloor \frac{n}{2} \rfloor}$. However, by definition,
this path corresponds to an edge $e$ of $\Gamma_{\lfloor \frac{n}{2} \rfloor}$ and is also a certifying path of $e$ in the 
$\Gamma_{\lfloor \frac{n}{2} \rfloor}$-immersion expansion in $W^{+}_{\lfloor \frac{n}{2} \rfloor}$. As the path contains internally a branch vertex of
$\Gamma_{\lfloor \frac{n}{2} \rfloor}$ we end up to a contradiction since we considered a 
$\Gamma_{\lfloor \frac{n}{2} \rfloor}$-strong-immersion model in $W^{+}_{\lfloor \frac{n}{2} \rfloor}$.

Let now $E$ be the set $\{e=\{u_{e},v_{e}\}\mid \text{two copies of } e \text{ belong to } W'\}$.
Notice that for every $e\in E$, one of its endpoints, say $u_{e}$ belongs to $\phi'(V(\Gamma_{\lfloor \frac{n}{2} \rfloor}))$. Then, it also holds that $u_{e}\in \phi'(V(H''))$.
 Recall that for every $e\in E$ the degree of $u_{e}$ in $W'$ is at most 3 (as it is an $H''$-immersion expansion). Since we are working on a strong immersion model of $H''$ 
all edges incident to $u_{e}$ belong to paths joining $u_{e}$ to other branch vertices. Then notice that by mapping each verger $\phi'^{-1}(u_{e})$ to $v_{e}$, 
$e\in E$ we may find a strong immersion model of $H''$ where the corresponding 
immersion expansion does not  contain multiple edges (Figure~\ref{fig:simple-wall-ex}).
Thus we obtain a strong immersion model of $H''$ in $W'$ such that its $H''$-immersion expansion is simple. 
As $n\leq 2|V(H)|$, we obtain that $H$ is a strong immersion of $W_{|V(H)|}$.
\end{proof}

By combining~\cite[Theorem 7]{Wollan15} with the main result of~\cite{Chuzhoy16impr} (see also~\cite{Chuzhoy15grid}) we can readily obtain the following.

\begin{theorem}\label{thm:gareggaer}
There  is a function $\tau:\N^{+}\rightarrow\N$ such that the following holds:
for every graph $G$ and $r\in\N^{+}$, if $\tcw(G)\geq f(r)$ then $W_{r}$  is an immersion of $G$. Moreover,   $f(r)=O(r^{29}\polylog(r))$.
\end{theorem}

\begin{lemma}
\label{y92ml1}
Let $G$ be a graph and let $H$ be a connected planar subcubic graph on $h$ vertices. Then $\tcw(G)=O(h^{29}\cdot (\pack_{{\cal I}(H)}(G))^{14.5}\cdot(\polylog(h)+\polylog(\pack_{{\cal I}(H)}(G)))$.
\end{lemma}

\begin{proof}
Let $\pack_{H}(G)\leq k$. Let $g(h,k)=f((h+1)\cdot \lceil(k+1)^{1/2}\rceil)$, 
where $f$ is the function of Theorem~\ref{thm:gareggaer}.
Suppose that $\tcw(G)\geq g(h)$. Then, from Theorem~\ref{thm:gareggaer}, we obtain that $G$ contains the wall $W$ of height 
$(h+1)\cdot \lceil(k+1)^{1/2}\rceil$ as an immersion. Notice that $W$ contains $k+1$ vertex-disjoint walls $W_{1},W_{2},\dots, W_{k+1}$ of height $h$.
From Lemma~\ref{lem:subcubicplanar}, each one of these walls contains $H$ as an immersion and thus an $H$-immersion expansion. Since, these walls
are vertex-disjoint they are also edge-disjoint. Hence, we have found a packing of $H$ of size $k+1>k$, a contradiction.
Therefore, $\tcw(G)\leq g(h,k)$. Notice now that, from Theorem~\ref{thm:gareggaer},
$g(h,k)=O(h^{29}k^{14.5}(\polylog (h)+\polylog (k))$ as required.
\end{proof}

The edge version of Theorem~\ref{main4u} follows as a corollary of~Theorem~\ref{f90op2c} and Lemma~\ref{y92ml1}.

%

\section{The vertex case}
\label{n9dh9w}

To prove the vertex version of Theorem~\ref{main4u}, is a much easier task. For this, we follow the same methodology
by using the graph parameter of treewidth instead of  tree-cut width, and topological minors instead of immersions.

\paragraph{\bf Treewidth.}
A graph $H$ is {\em $k$-chordal} if it does not contain any induced cycle of length at least $4$ and no clique one more than $k+1$ vertices. The {\em treewidth} 
of a graph $G$ is the minimum $k$ for which $G$ is a subgraph of a $k$-chordal graph.

\paragraph{\bf Topological Minors.}
Let $H$ and $G$ graphs. Similarly to the definition of immersions, we say that $G$ contains $H$ as a
\emph{topological minor} if there is pair of functions $(\phi, \psi)$, called
$H$-\emph{topological-minor model} such that $\phi$ is an injection of $V(H) \to V(G)$ and
$\psi$ sends $\{u,v\}_i$ to a path of $G$ between
$\phi(u)$ and $\phi(v)$, for every $\{u,v\}\in E(H)$ and every $i \in
\intv{1}{\mult_H(\{u,v\})}$, in a way such that
distinct edges are sent to internally vertex-disjoint paths. Every vertex in the image of $\phi$
is called a \emph{branch vertex}. Observe that if $(\phi,\psi)$ is an $H$-topological-minor model in $G$ then 
$(\phi,\psi)$ is  an $H$-strong-immersion model in $G$.

For the proof of the vertex case of Theorem~\ref{main4u},
 we require the following two ``vertex counterparts'' 
of Theorem~\ref{f90op2c} and~Lemma~\ref{y92ml1} respectively.

\begin{proposition}
\label{ropjkg}
Let ${\cal H}$ be a class of connected graphs and let $t$ be a non-negative  integer.
Then ${\cal H}$ has the {\sf v}-E{\sf \footnotesize  \&}P property for the graphs of treewidth at most $t$
with a gap that is a polynomial function on $t$.
\end{proposition}
\begin{lemma}
\label{y92smlx1}
Let $G$ be a graph and let $H$ be a connected planar  graph on 
$h$ vertices and without any ${\cal I}(H)$-vertex packing of size greater than $k$.
Then $\tw(G)=(h\cdot k)^{O(1)}$.
\end{lemma}

Proposition~\ref{ropjkg}  was 
proven by Thomassen in~\cite{Thomassen88onth} (see also~\cite{FominST11stre,ChekuriC13larg}).
For Lemma~\ref{y92smlx1}, we need the fact that there 
is a polynomial function $\lambda:\N^{+}\rightarrow\N$ such that for 
every $r\in \N^{+}$, every graph with treewidth at  least $\lambda(r)$ contains 
$W_{r}$ as a topological minor. The existence of such a function $\lambda$
follows from the {\em grid exclusion theorem}
of Robertson and Seymour in~\cite{RobertsonS86GMV} (see also~\cite{RobertsonST94,Diestel05grap}) and 
the polynomiality of $\lambda$ was proved recently by Chekuri and Chuzhoy in~\cite{ChekuriC13poly} (see also~\cite{Chuzhoy15grid,Chuzhoy16impr} for improvements).
Then~Lemma~\ref{y92smlx1} can be proved using the same arguments as in Lemma~\ref{y92ml1},
taking 
into account that~Lemma~\ref{lem:subcubicplanar} also holds if we consider 
topological minors instead of immersions and the fact that a
topological minor model is also an immersion model.

\noindent The vertex version of Theorem~\ref{main4u} follows from~Proposition~\ref{ropjkg} and Lemma~\ref{y92smlx1} if, in Proposition~\ref{ropjkg}, we set ${\cal H}={\cal I}(H)$ and $t=(h\cdot k)^{O(1)}$.

\section{Discussion}
\label{d0opo45iuyt}

Notice that in~Theorem~\ref{main4u} we demand that $H$ is a connected graph. It is easy to extend this result
if instead of $H$ we consider some 
collection ${\cal H}$ is  of connected graphs containing one that is planar subcubic and where ${\cal I}(H)$
contains all graphs containing some graph in ${\cal H}$ as an immersion.
Moreover, it is easy 
to drop the connectivity condition for the vertex variant using arguments from~\cite{RobertsonS86GMV}.
However it remains open whether this can be done for the edge variant as well.

Naturally, the most challenging 
problem on the Erdős–Pósa properties of immersions
is to characterize the graph classes:

$${\cal H}^{\sf v/e}=\{H\mid {\cal I}(H)\mbox{\ has the {\sf v/e}-E{\sf \footnotesize  \&}P property}\}$$

\noindent In this paper we prove that both  ${\cal H}^{\sf v}$  and  ${\cal H}^{\sf e}$ contain all planar subcubic graphs.
It is an interesting question  whether ${\cal H}^{\sf v/e}$ are wider than this.
Using arguments similar to~\cite{Raymond2013edge,RobertsonS86GMV} it is possible to 
prove the following. 

\begin{lemma}\label{to094rhfk}
None of ${\cal H}^{\sf v}$ and ${\cal H}^{\sf e}$ contains a non-planar  subcubic graph.
\end{lemma}

 \begin{figure}[h]
   \centering
 \centering
\scalebox{0.2}{\includegraphics{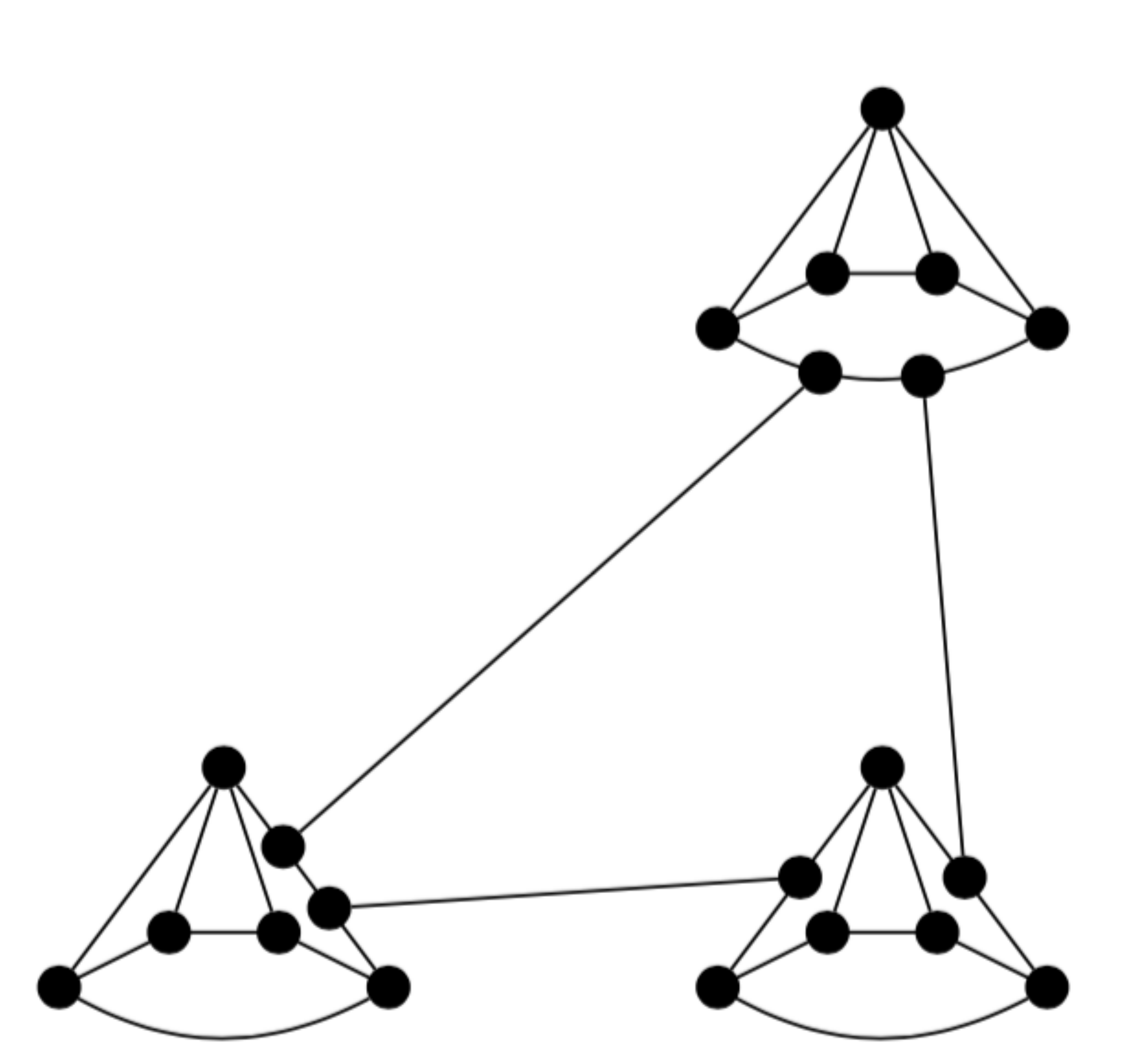}}
   \caption{A biconnected graph $H$ for which
${\cal I}(H)$ does not have  the {\sf v/e}-E{\sf \footnotesize  \&}P property.}
   \label{fig:ce}
 \end{figure}

Actually, the arguments of~\cite{Raymond2013edge,RobertsonS86GMV} permit
to exclude all non-planar graphs from  ${\cal H}^{\sf v}$.
For the non-subcubic case, we can first observe that $K_{1,4}$, which is planar and non-subcubic belongs 
in both  ${\cal H}^{\sf v}$  and   ${\cal H}^{\sf e}$. However, this is not the case for all  planar and non-subcubic
graphs as is indicated in the following  observation.

\begin{observation}\label{tso094rshfk}
There exists  a 3-connected  non-subcubic graph $H$ that  does belong neither to ${\cal H}^{\sf v}$  nor  to  ${\cal H}^{\sf e}$.\end{observation}

%

\begin{figure}[h]
  \centering
\scalebox{0.2}{\includegraphics{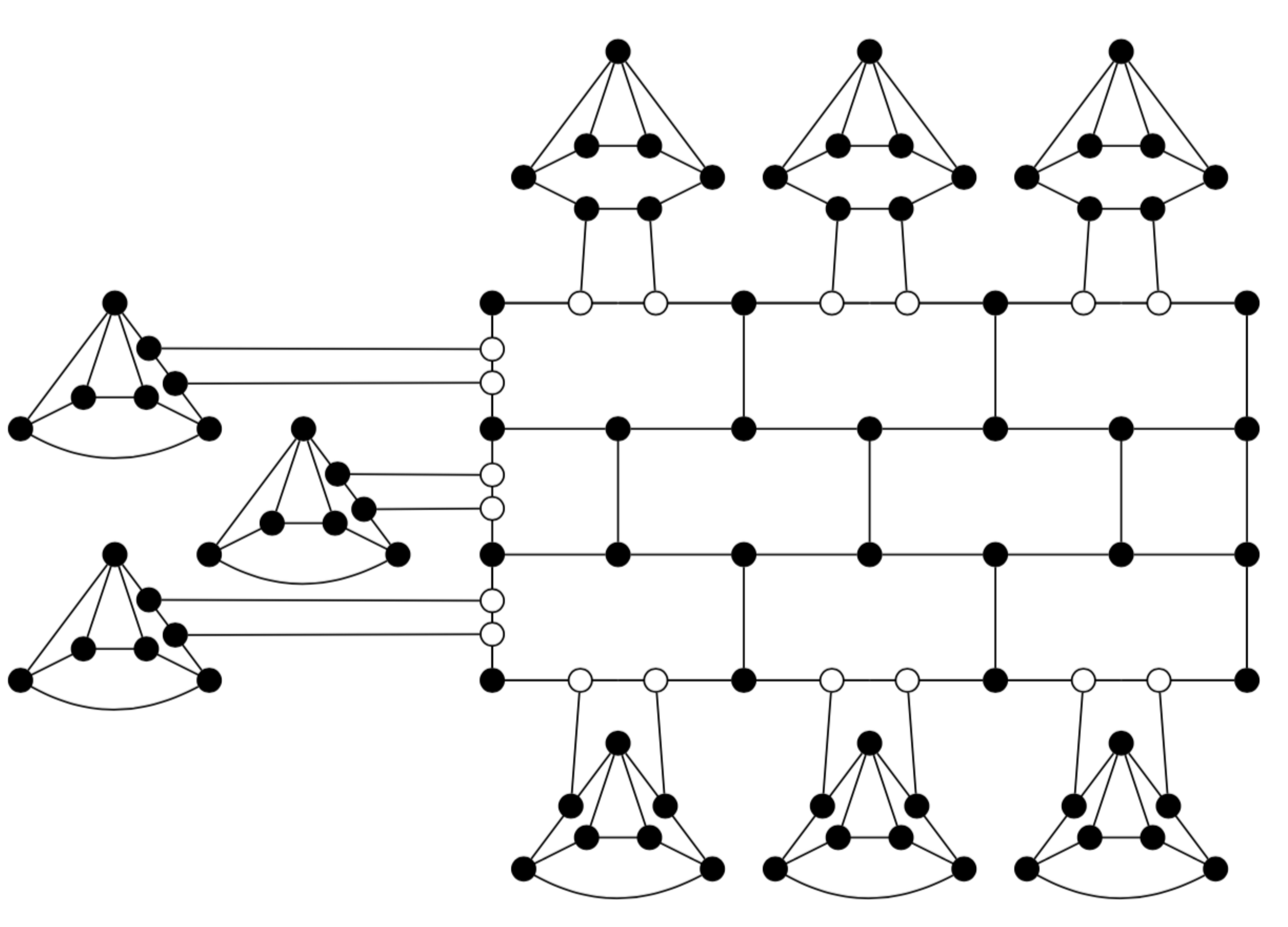}}
  \caption{The host graph $G$.}
  \label{fig:host}
\end{figure}

\begin{proof}
Thomassen in~\cite{Thomassen88onth} provided an example of a tree
that does not belong in  ${\cal H}^{\sf v}$ (the same graph
does not belong in  ${\cal H}^{\sf e}$ either). 
Inspired by the construction of~\cite{Thomassen88onth}, we consider first the
graph  $H$ is  
depicted  in Figure~\ref{fig:ce}.
To see that $H\not\in {\cal H}^{\sf v}$ and $H\not\in {\cal H}^{\sf e}$,  consider as host graph $G$ the graph 
in Figure~\ref{fig:host}. This graph consists of a main body that is a wall of height 3 and three 
triples of graphs attached at its upper, leftmost, and lower paths.  Each of these triples consists 
of three copies of some of the 3-connected components of $H$.
Notice that $G$ does not contain more than one {$H$-immersion expansion}.
However,  in order to cover all $H$-immersion expansions of $G$ one needs to remove at least 3 edges/vertices.
By increasing the heigh of the wall of $G$, we may increase the minimum size of 
an ${\cal I}(H)$-vertex/edge cover while no ${\cal I}(H)$-vertex/edge 
packing of size greater than $1$ will appear. It is easy to modify $H$ so to make
it 3-connected: just add a new vertex and make it adjacent 
with the tree vertices of degree 4. The resulting graph $H'$ remains 
planar. The same arguments, applied to an easy modification of the host graph, can prove that $H'$ 
is not a graph in ${\cal H}^{\sf v}$ or ${\cal H}^{\sf e}$.
\end{proof}

Providing an exact characterization of ${\cal H}^{\sf v}$  and  ${\cal H}^{\sf e}$ is an insisting open problem.
A first step to deal with this problem  could be  the cases of $\theta_{4}=\tikz[every node/.style={black node, minimum size=4pt},rotate=30, scale =0.3]{
  \node (x) at (0,0) {};
  \node (y) at (2,0) {};
  \draw (x) .. controls (1, 0.75) and (1, 0.75) .. (y);
  \draw (x) .. controls (1, -0.75) and (1, -0.75) .. (y);
  \draw (x) .. controls (1, .25) and (1, .25) .. (y);
  \draw (x) .. controls (1, -.25) and (1, -.25) .. (y);}$ and the 4-wheel $\tikz[every node/.style={black node, minimum size=4pt}, rotate = 45, scale =0.25]{
  \node (x) at (0,0) {};
  \draw (1,0) node (y1) {} -- (0,1) node (y2) {} -- (-1,0) node (y3) {} -- (0,-1) node (y4) {} -- cycle;
  \draw (x) -- (y1) (x) -- (y2) (x) -- (y3) (x) -- (y4);}$. Especially for the 4-wheel,
the structural results in~\cite{BelmonteGLT2016thes} might be useful in this direction.

%

\end{document}